\documentclass[11pt]{amsart}
\usepackage{fullpage}

\usepackage{url, 
  amssymb,setspace, mathrsfs,fontenc}
\usepackage[alphabetic]{amsrefs}
\usepackage{graphicx}
\usepackage[mathcal]{euscript}
\usepackage{verbatim}
\usepackage{hyperref}
\hypersetup{backref=true}

\newtheorem{thm}{Theorem}
\newtheorem{lem}[thm]{Lemma}

\newtheorem{prop}[thm]{Proposition}

\newtheorem{cor}[thm]{Corollary}
\newtheorem{defe}[thm]{Definition}
 
\theoremstyle{remark}
\newtheorem{rem}[thm]{Remark}
\newtheorem{exam}[thm]{Example}
\newtheorem{ntn}[thm]{Convention}

\DefineSimpleKey{bib}{myurl}
\newcommand\myurl[1]{\url{#1}}
\BibSpec{webpage}{
  +{}{\PrintAuthors} {author}
  +{,}{ \textit} {title}
  +{}{ \parenthesize} {date}
  +{,}{ \myurl} {myurl}
}

\let\oldmarginpar\marginpar
\renewcommand\marginpar[1]{\-\oldmarginpar[\raggedleft\footnotesize #1]%
{\raggedright\footnotesize #1}}

\newcommand{\nc}{\newcommand}

\nc{\ssec}{\subsection}

\nc{\on}{\operatorname}

\nc {\cG} {\mathcal{G}}
\nc {\cK}{\mathcal{K}}
\nc {\cC} {\mathcal{C}}
\nc {\cL} {\mathcal{L}}
\nc {\cE} {\mathcal{E}}
\nc {\cM} {\mathcal{M}}
\nc {\cO}{\mathcal{O}}
\nc {\cF}{\mathcal{F}}
\nc {\cZ}{\mathcal{Z}}

\nc {\bZ}{\mathbb{Z}}
\nc {\bQ}{\mathbb{Q}}

\nc {\uG} {\underline{G}}
\nc {\cB}{\mathcal{B}}
\nc{\rat}{\mathrm{rat}}
      
\nc {\fk}{\mathfrak{k}}
\nc {\fI}{\mathfrak{i}}
\nc {\fg} {\mathfrak{g}}
\nc {\fu} {\mathfrak{u}}
\nc {\fl} {\mathfrak{l}}
\nc {\fn} {\mathfrak{n}}
\nc {\cP} {\mathcal{P}}
\nc {\fz} {\mathfrak{z}}
\nc {\fc}{\mathfrak{c}}
\nc {\fh}{\mathfrak{h}}
\nc {\fp}{\mathfrak{p}}
\nc {\fC}{\mathcal{C}}
\nc {\ft}{\mathfrak{t}}

\nc{\tg} {\mathtt{g}}
\nc {\hfg} {\widehat{\fg}}
\nc {\hG} {\check{G}}

\nc {\bGm} {\mathbb{G}_m}
\nc{\bC}{\mathbb{C}}
\nc{\bV}{\mathbb{V}}
\nc{\bP}{\mathbb{P}}
\nc{\bA}{\mathbb{A}}

\nc{\Sl}{\mathfrak{sl}}

\nc{\ra}{\rightarrow}

\nc {\tU}{\tilde{U}}
\nc {\tSym}{\widetilde{Sym}}

\nc {\Bun}{\mathrm{Bun}}
\nc {\cA}{\mathcal{A}}
\nc {\Fun}{\mathrm{Fun}}
\nc {\crit}{\mathrm{crit}}
\nc {\Ind}{\mathrm{Ind}}
\nc {\Vac}{\mathrm{Vac}}
\nc {\gr}{\mathrm{gr}}
\nc {\ad}{\mathrm{ad}}
\nc {\Sym}{\mathrm{Sym}}
\nc {\Ram}{\mathrm{Ram}}
\nc {\FG}{\mathrm{FG}}
\nc {\Op}{\mathrm{Op}}
\nc {\Hitch}{\mathrm{Hitch}}
\nc {\fb}{\mathfrak{b}}
\nc{\cDt}{\mathcal{D}^\times}
\nc{\cDb}{\mathcal{D}_b^\times}
\nc{\cDbp}{\mathcal{D}_{b'}^\times}
\nc {\gl}{\mathfrak{gl}}
\nc {\Sp}{\mathfrak{sp}}
\nc {\So}{\mathfrak{so}}
\nc {\bR}{\mathbb{R}}
\nc {\Hom}{\mathrm{Hom}}
\nc {\Irr}{\mathrm{Irr}}
\nc {\Id}{\mathrm{Id}}
\nc {\Tr}{\mathrm{Tr}}
\nc {\rk}{\mathrm{rank}}
\nc {\rank}{\mathrm{rank}}
\nc {\cW}{\mathcal{W}}
\nc {\cI}{\mathcal{I}}
\nc {\Fr}{\mathrm{Fr}}
\nc {\ff}{\mathfrak{f}}

\nc {\LocSys}{\mathrm{LocSys}}
\nc {\Ga}{\mathrm{Ga}}
\nc {\ord}{\mathrm{ord}}
\nc{\pole}{\mathrm{pole}}

\newcommand{\C}{\mathbb{C}}

\newcommand{\dtt}{\frac{dt}{t}}

\newcommand{\n}{\nabla}

\newcommand{\nbr}{[\nabla]}

\newcommand{\Om}{\Omega}
\newcommand{\Ad}{\mathrm{Ad}}

\renewcommand{\a}{\alpha}
\renewcommand{\b}{\beta}
\newcommand{\fs}{\mathfrak{s}}
\newcommand{\Waff}{W^{\mathrm{aff}}}

\newcommand{\Gm}{\mathbb{G}_m}

\DeclareMathOperator{\GL}{GL}
\DeclareMathOperator{\SL}{SL}
\DeclareMathOperator{\SP}{Sp}
\DeclareMathOperator{\SO}{SO}
\DeclareMathOperator{\Or}{O}

\DeclareMathOperator{\Spec}{Spec}

\DeclareMathOperator{\Res}{\mathrm{Res}}

\begin{document} 
\title{Rigid connections on $\bP^1$ via the Bruhat-Tits building} 
\author{Masoud Kamgarpour}
\author{Daniel S. Sage} 


\address{School of Mathematics and Physics, The University of Queensland, Brisbane, QLD 4072, Australia} 
\email{masoud@uq.edu.au}

\address{Department of Mathematics, Louisiana State University, Baton
  Rouge, LA 70808, USA} 
\email{sage@math.lsu.edu}

\date{\today}

\begin{abstract} We apply the theory of fundamental strata of Bremer
  and Sage to find cohomologically rigid $G$-connections on the
  projective line, generalising the work of Frenkel and Gross. In this
  theory, one studies the leading term of a formal connection with
  respect to the Moy-Prasad filtration associated to a point in the
  Bruhat-Tits building. If the leading term is regular semisimple with
  centraliser a (not necessarily split) maximal torus $S$, then we
  have an $S$-toral connection.  In this language, the irregular
  singularity of the Frenkel-Gross connection gives rise to the
  homogeneous toral connection of minimal slope associated to the
  Coxeter torus $\cC$. In the present paper, we consider connections
  on $\Gm$ which have an irregular homogeneous $\cC$-toral singularity
  at zero of slope $i/h$, where $h$ is the Coxeter number and $i$ is a
  positive integer coprime to $h$, and a regular singularity at
  infinity with unipotent monodromy.  Our main result is the
  characterisation of all such connections which are rigid.
\end{abstract}

\maketitle 

\tableofcontents

\section{Introduction}

\subsection{Rigid connections} 
The study of rigid connections originates in the seminal work of
Riemann on Gauss's hypergeometric function \cite{Riemann}. His central
insight was that one should study Gauss's function by the
corresponding \emph{local system} on $\bP^1_\bC-\{0,1,\infty\}$,
defined as the system of holomorphic local solutions of the
hypergeometric differential equation.  This approach greatly clarified
and expanded the invariant properties of the hypergeometric functions
established by Euler, Gauss, and Kummer. Riemann's investigation was a
stunning success largely because the hypergeometric local system is
\emph{physically rigid}, i.e., determined up to isomorphism as soon as
one knows the local monodromies at the three singular points.

In modern times, the subject of rigid local systems was rejuvenated by
the influential book of Katz \cite{KatzRigid}. Katz took the point of
view that a rigid rank $n$ local system on $\bP^1-\{\textrm{$m$
  points}\}$ corresponds to an $n\times n$ first order system of
differential equations with singularities at the $m$ missing points.
He then proceeded to systematically study and classify these local
systems using the notion of middle convolution. To facilitate the
understanding of rigidity, he defined the notion of
\emph{cohomological rigidity} which, roughly speaking, means that the
local system has no infinitesimal deformations. Using deep
results of Laumon on the $\ell$-adic Fourier transform, he proved that
physical and cohomological rigidity for $\ell$-adic local systems are
equivalent. Subsequently, Bloch and Esnault \cite{BE} proved an
analogous result in the de Rham setting. Cohomological rigidity gives
rise to a practical numerical criterion for rigidity, see \S
\ref{ss:numerical}.

\subsection{Rigid $G$-connections} 
One can view an $n\times n$ first order system of differential
equation as a connection on a rank $n$ vector bundle which, in our
case, lives on $\bP^1-\{\textrm{$m$ points}\}$. Using the equivalence
between rank $n$ vector bundles and principal $\GL_n$-bundles, one
obtains a principal $\GL_n$-bundle equipped with a connection, or, for brevity, a \emph{$\GL_n$-connection}. This construction can be generalised to any algebraic group $G$, thereby obtaining \emph{$G$-connections}.

Foundational results on $G$-connections were established by Babbit and Varadarajan \cite{BV}. More recently, this subject has received renewed attention due to its application to the geometric Langlands program. Indeed, if $G$ is reductive,  $G$-connections on complex algebraic curves are precisely the \emph{geometric Langlands parameters} \cite{FrenkelBook}.\footnote{Another reason for interest in $G$-connections is that they arise naturally in $G$-analogues of non-abelian Hodge theory and the Riemann--Hilbert correspondence, cf. \cite{Boalch} where Bruhat-Tits theory also makes an appearance.} In particular, the geometric Langlands community has been interested in \emph{rigid} $G$-connections because it is expected that in this case the Langlands correspondence can be described explicitly. Another source of interest in rigid $G$-connections is the expectation that they should be, in some sense, motivic. For instance, in the $\GL_n$ case, it is known that an irreducible connection is rigid if and only if it can be reduced to the trivial connection using Fourier transforms and convolutions \cite{KatzRigid, Arinkin}. 

\begin{rem} 
Physical and cohomological rigidity make sense for $G$-connections, and it is known that physical rigidity implies cohomological rigidity; however, it is not known if the converse implication holds \cite[\S 3.2]{YunCDM}. Thus, while we have many examples of cohomologically rigid $G$-connections, many of which are conjectured to be physically rigid, we are unaware of any (proven) example of a physically rigid $G$-connection beyond type $A$. 
\end{rem} 

\begin{ntn} Henceforth, we will simply call cohomologically rigid connections rigid. 
\end{ntn}

\subsection{Constructions of rigid $G$-connections} We now review some of the methods for constructing and analysing rigid $G$-connections. 

\subsubsection{Opers} 
Roughly speaking, a $G$-oper is a $G$-connection whose underlying
bundle admits a reduction to a Borel subgroup satisfying a certain
transversality condition with respect to the connection \cite[\S
4]{FrenkelBook}. Using the notion of oper and building on the earlier
work of Deligne \cite{ExponentialSums} and Katz \cite{KatzKloosterman,
  KatzRigid}, Frenkel and Gross \cite{FG} constructed a remarkable
$G$-connection $\nabla_\FG$ on the trivial $G$-bundle on
$\bP^{1}-\{0,\infty\}$ satisfying the following properties:
\begin{enumerate} 
\item it has a regular singularity at $0$ with principal unipotent monodromy;
\item it has an irregular singularity  at $\infty$ with slope $\frac{1}{h}$, where $h$ is the Coxeter number;
\item it is cohomologically rigid;
\item it has a large differential Galois group (equal to $G$ in many cases). 
\end{enumerate} 

\subsubsection{Rigid automorphic data} The geometric Langlands program predicts a correspondence between $G$-local systems and certain automorphic data for the Langlands dual group $\hG$. One expects therefore that there is a notion of rigidity for automorphic data which corresponds to rigidity for local systems. The first example of rigid automorphic data was found by Heinloth, Ng\^o, and Yun \cite{HNY}, who used this to define the analogue of Kloosterman sheaves for general reductive groups. 
They conjectured that the de Rham version of their local system is isomorphic to the Frenkel--Gross connection $\nabla_\FG$. This conjecture was subsequently proved by \cite{Zhu}. Lam and Templier \cite{LT} then used results of \cite{HNY, Zhu} to settle mirror symmetry for minuscule flag varieties. 

Yun subsequently systematised the study of rigid automorphic data \cite{YunGalois, Yun, YunCDM}. In particular, in \cite{YunGalois}, he settled an old question of Serre on inverse Galois theory by constructing sought-after Galois representations via rigid automorphic data, again highlighting the role of rigidity in the Langlands program and related areas.  

\subsubsection{$\theta$-connections} A torsion automorphism $\theta\in \mathrm{Aut}(\fg)$ gives rise to a grading 
\begin{equation}\label{eq:torgrad}
\fg=\fg_0\oplus \cdots \fg_{m-1}. 
\end{equation}
The action of $G_0$ on $\fg_1$ and the resulting GIT quotient was
studied extensively by Vinberg and his school under the name of
\emph{$\theta$-groups}.  Yun and Chen \cite{Yun, Chen} associate to a
stable element $X\in \fg_1$ a twisted connection called a
\emph{$\theta$-connection}.\footnote{In fact, the constructions given
  by Yun and Chen are different. Conjecturally, they result in
  equivalent connections.} In general, $\nabla_X$ is a connection on
$\bGm$ which is regular singular with unipotent monodromy at $0$ and
irregular at $\infty$. If the grading is inner, so that the connection
is untwisted, then the slope at the irregular point equals $1/m$. In
\cite{Chen}, it is shown that these connections are usually rigid.
For an appropriate choice of $\theta$ (specifically the one defined by
the torsion element $\exp(\check{\rho}/h)$ in a fixed maximal torus
$T$), the corresponding $\theta$-connection $\nabla_X$ equals the
Frenkel-Gross connection $\nabla_{\FG}$.

\subsection{Our goal} 
In this paper, we construct a large class of connections satisfying
properties analogous to the Frenkel-Gross connection. We found these
connections by using the theory of fundamental strata for formal
connections introduced by Bremer and Sage
\cite{BremerSageisomonodromy, BremerSagemodspace, BSregular,
  BSminimal, Sageisaac}.  A major theme in these works is that
Moy-Prasad filtrations -- certain filtrations on the loop algebra
associated to points in the Bruhat-Tits building -- can be used to
define a new notion of a leading term for a formal
connection. Moreover, if this leading term is regular semisimple  in a
suitable sense, then the connection enjoys favourable properties.
If the centraliser of the leading term is the (not necessarily split)
maximal torus $S$, we say that the connection is $S$-toral.

Given a point $x$ in the Bruhat-Tits building $\cB$, there is an
associated Moy-Prasad filtration \cite{MP} on the loop algebra
$\fg(\bC(\!(t)\!))$, and one can speak of the leading term of $\nabla$
with respect to this filtration.  If we restrict to the standard
apartment (with respect to a fixed maximal torus $T\subset G$), then
this filtration comes from a grading on the polynomial loop algebra
$\fg(\bC[t, t^{-1}])$.  We say that $\nabla=d+A\,dt$ is
\emph{homogeneous} if $A$ is homogeneous with respect to the grading. In
this case, $A\in \fg(\bC[t, t^{-1}])$, and we can think of $\nabla$ as
a connection on the trivial bundle on $\bP^1-\{0,\infty\}$. We can
then use properties of the leading term to investigate whether
$\nabla$ is rigid.

If $x$ is chosen to be the barycentre of the fundamental alcove, then
the resulting filtration comes from the principal
grading.\footnote{This grading already appeared in the works of Kac on
  infinite dimensional Lie algebras \cite{Kac}.} Our main result is
the classification of all rigid connections that are homogeneous with
respect to the principal grading and have a regular semisimple leading
term.  The corresponding formal connections at
the irregular singularity are toral with respect to a certain elliptic
maximal torus called the Coxeter torus.   As we shall see, there are two different flavours of such rigid
connections depending on whether the slope at the irregular point is
less than or bigger than one.

If the slope is bigger than one, then we can determine the rigid
connections in a uniform manner for all simple groups $G$. These
connections are regular everywhere except at $0$, where they
have an irregular singularity of slope $(h+1)/h$. Explicitly, they are defined as follows. For each negative simple root $\alpha$, let $x_\alpha$ be a generator of the corresponding root space. 
Let $N:=\sum x_\alpha$ and let $E$ be a generator of the highest root
space. The connection in question is then given by\footnote{In Type A,
  the $\ell$-adic analogue of $\nabla$ has been studied intensively by
  Katz,  cf. \cite{KatzFinite}. Finding the $\ell$-adic analogue of $\nabla$ for general $G$ is an interesting open problem.}  
\begin{equation} 
\nabla:= d + (N+t^{-1}E)\frac{dt}{t^2}. 
\end{equation} 
Note that replacing $t^2$ by $t$ gives the Frenkel--Gross connection.

On the other hand, if
the slope is less than one, we determine rigidity by a
case-by-case analysis.  This is necessary because, in contrast to the
Frenkel-Gross situation, the monodromy at the regular singular point
is no longer the principal unipotent class.\footnote{In particular,
  the connections studied here are not necessarily in oper form.}
Therefore, specific information about nilpotent conjugacy classes in
each type is required. It turns out that if $G$ is classical, such
rigid connections are plentiful. In contrast, if $G$ is exceptional,
there is only one example aside from the Frenkel-Gross connection, and
it appears for $G=E_7$. This connection is irregular at $0$ of slope
$7/18$ and is regular singular at $\infty$ with unipotent monodromy of
type $A_2+3A_1$.

\subsubsection{Relationship to $\theta$-connections} As already noted,
$\theta$-connections have slope $1/m$ at the irregular point, whereas
the connections considered here have slopes $r/h$ , with $r$ a
positive integer relatively prime to $h$.  However, there is a strong
relationship between our construction and the approach of Yun and
Chen.  On the one hand, our methods can be used to construct
connections equivalent to those of Yun and Chen. Indeed, in our
language, (formal) $\theta$-connections are toral connections of
minimal depth associated to ``elliptic regular'' maximal tori in the
loop group.\footnote{One needs to extend the theory of toral
  connections to twisted loop groups to get those $\theta$-connections
  which are twisted connections.}  In the Frenkel-Gross case (the
simplest example of a $\theta$-connection), one obtains a toral
connection of minimal depth with respect to the Coxeter torus.  On the
other hand, the connections studied in this paper can also be obtained
by considering the principal torsion automorphism.  In future work, we
will examine a common generalisation of the work of Yun and Chen and
of this paper -- the classification of homogeneous rigid connections
whose irregular singularity is toral for other classes of elliptic
regular maximal tori.  These would correspond to variants of
$\theta$-connections where one starts with appropriate homogeneous
elements of the grading \eqref{eq:torgrad} of degree $r>1$ relatively
prime to $m$.

\subsection{Organization of the paper} In \S \ref{s:BT}, we give a
quick review of Moy-Prasad filtrations and their role in studying
formal connections, following earlier works of Bremer and Sage. In \S
\ref{s:Coxeter}, we discuss formal Coxeter connections. In particular,
we write convenient expressions for these connections and determine
such basic properties as slope, adjoint irregularity, and the local
differential Galois group. In \S \ref{s:Rigidity}, we state and prove
our main result giving the necessary and sufficient conditions for
(global) Coxeter connections to be rigid. We conclude the paper by
discussing the global differential Galois group of these connections.

\subsection{Acknowledgements} We would like to thank Willem de Graaf
for his generous help regarding questions about nilpotent orbits in
exceptional groups. We also thank the Centre International de
Rencontres Math\'ematiques (CIRM) in Luminy for hosting us during the
early stages of this project. MK is supported by an
Australian Research Council Discovery grant. DS is supported by an NSF
grant and a Simons Foundation grant.

 \section{Moy-Prasad filtrations and formal connections} \label{s:BT}
 Let $G$ be a simple complex algebraic group with Lie algebra
 $\fg$. Let $B$ be a  Borel subgroup, and $T\subset B$ a maximal torus with Lie algebras $\ft$ and $\fb$. We write $\fg_\alpha$ for the root space corresponding to a root $\alpha$.  
 Let $\cK=\bC(\!(t)\!)$ be the field of Laurent series with ring of integers $\cO=\bC[\![t]\!]$. Let $G(\cK)$ and $\fg(\cK)$ denote the loop group and algebra, respectively.  
 
 \subsection{Fundamental strata}
 Let $\cB$ be the Bruhat-Tits building of $G$; it is
a simplicial complex whose facets are in bijective correspondence with
the parahoric subgroups of the loop group $G(\cK)$.  The standard apartment $\cA$
associated to the split maximal torus $T(\cK)$ is an affine space
isomorphic to $X_*(T)\otimes_\bZ \bR$.  Given $x\in\cB$, we denote by
$G(\cK)_x$ (resp. $\fg(\cK)_x$) the parahoric subgroup (resp.
subalgebra) corresponding to the facet containing $x$.

For any $x\in\cB$, the Moy-Prasad filtration associated to $x$ is a
decreasing $\bR$-filtration 
\[
\{\fg(\cK)_{x,r}\mid
r\in\bR\}
\] of $\fg(\cK)$ by $\cO$-lattices.  The filtration satisfies
$\fg(\cK)_{x,0}=\fg(\cK)_x$ and is periodic in the sense that
$\fg(\cK)_{x,r+1}=t\fg(\cK)_{x,r}$.  Moreover, if we set
$\fg(\cK)_{x,r+}=\bigcup_{s>r}\fg(\cK)_{x,s}$, then the set of $r$ for
which $\fg(\cK)_{x,r}\ne\fg(\cK)_{x,r+}$ is a discrete subset of
$\bR$.

For our purposes, it will suffice to give the explicit definition of the filtration for
$x\in\cA$.  In this case, the filtration is determined by a grading on
$\fg(\bC[t,t^{-1}])$, with the graded subspaces given by
\begin{equation*}
\fg(\cK)_{x}(r)=\begin{cases} \ft t^r\oplus
  \bigoplus\limits_{\a(x)+m=r}\fg_\a t^m, &\text{if } r\in\bZ\\\\
\bigoplus\limits_{\a(x)+m=r}\fg_\a t^m,&\text{otherwise.}
\end{cases}
\end{equation*}

\begin{rem} The gradings associated to $x\in \cA_\bQ$ appeared in early works of Kac on infinite dimensional algebras. Thus, these  are sometimes called  Kac-Moy-Prasad gradings, cf. \cite{Chen}. 
\end{rem}

Let $\kappa$ be the Killing form for $\fg$.  Any element $X\in
\fg(\cK)$ gives rise to a continuous $\C$-linear functional on
$\fg(\cK)$ via $Y\mapsto \Res\kappa(Y,X)\dtt$.  This identification
induces an isomorphism
\begin{equation*} 
(\fg(\cK)_{x,r}/\fg(\cK)_{x,r+})^\vee\cong \fg(\cK)_{x,-r}/\fg(\cK)_{x,-r+}.
\end{equation*}

\begin{defe} A \emph{$G$-stratum of depth $r$} is a triple $(x,r,\b)$ with $x\in\cB$,
$r\ge 0$, and $\b\in(\fg(\cK)_{x,r}/\fg(\cK)_{x,r+})^\vee$. We say that this strata is \emph{based} at $x$. 
\end{defe} 

 Any
element of the corresponding $\fg(\cK)_{x,-r+}$-coset is called a
representative of $\b$. If $\hat{\beta}$ is a representative, we will abuse notation to refer to the stratum as $(x,r,\hat{\beta})$.     If $x\in\cA$, there is a unique homogeneous
representative $\b^\flat\in \fg(\cK)_{x}(-r)$. 

\begin{defe} The stratum is called
\emph{fundamental} if every representative is non-nilpotent. 
\end{defe} 

 When
$x\in\cA$, it suffices to check that $\b^\flat$ is non-nilpotent.

\subsection{Leading term of formal connections} 
 A formal flat
$G$-bundle $(\cE,\n)$ (or simply a formal $G$-connection) is a principal $G$-bundle $\cE$ on the formal punctured disk $\cDt:=\Spec(\cK)$
endowed with a connection $\n$ (which is automatically flat).  Upon
choosing a trivialisation, the connection may be written in terms of
its matrix 
\[
\nbr_\phi \in \Om^1_{F}(\fg(\cK))
\]
 via $\n=d+\nbr_\phi$.
If one changes the trivialisation by an element $g\in G(\cK)$, the
matrix changes by the \emph{gauge action}:
\begin{equation}\label{gauge}
\nbr_{g \phi} = \mathrm{Gauge}_g(\nbr_\phi)= \Ad_g (\nbr_\phi) - (dg) g^{-1}.
\end{equation}
Accordingly, the set of isomorphism classes of flat $G$-bundles on
$\cDt$ is isomorphic to the quotient $\fg(\cK)\dtt/G(\cK)$, where the
loop group $G(\cK)$ acts by the gauge action.

\begin{defe} If $x\in\cA\cong\ft_\bR$, we say that $(\cE,\n)$ contains the stratum $(x,r,\b)$
with respect to the trivialisation $\phi$ if $[\n]_\phi/(dt/t)-x\in
\fg(\cK)_{x,-r}$ and is a representative for $\b$.  
\end{defe}

We refer the reader to 
\cite{BSminimal} for the definition for an arbitrary point $x\in \cB$. If $\nabla$ contains the stratum $(x,r,\b)$, then we can think of $\b$ as the leading term of $\nabla$ with respect to $x$. The following theorem shows that a non-nilpotent leading term contains meaningful information about $\nabla$:\footnote{This theorem and
  Theorem~\ref{formaltypes} remain true for connected reductive $G$.}

\begin{thm}[{\cite[Theorem 2.14]{BSminimal}}] \label{t:fundamental} 
 Every flat $G$-bundle $(\cE, \n)$ contains a fundamental
  stratum $(x,r,\b)$, where $x$ is in the closure of the fundamental alcove
  $C\subset\cA$ and $r\in\bQ$; the
  depth $r$ is positive if and only if $(\cE, \n)$ is irregular
  singular.  Moreover, the following statements hold.
\begin{enumerate}
\item If $(\cE, \n)$ contains the stratum $(y,r', \b')$, then
$r' \ge r$.  
\item  If $(\cE,\n)$ is irregular singular, a stratum
  $(y,r', \b')$ contained in $(\cE, \n)$ is fundamental if and only if
  $r' = r$.
\end{enumerate}
\end{thm}
As a consequence, one can define the slope of $\n$ as the depth of any
fundamental stratum it contains.

\subsection{Regular strata and toral flat $G$-bundles}\label{ss:toral}

We will also need some results on flat $G$-bundles which contain a
\emph{regular stratum}, a kind of stratum that satisfies a graded
version of regular semisimplicity.  For convenience, we will only
describe the theory for strata based at points in $\cA$.

Let $S\subset G(\cK)$ be a (in general, non-split) maximal torus, and
let $\fs \subset \fg(\cK)$ be the associated Cartan subalgebra.  We
denote the unique Moy-Prasad filtration on $\fs$ by $\{\fs_r\}$.  More
explicitly, we first observe that if $S$ is split, then this is just
the usual degree filtration.  In the general case, if $\cK_b$ is a
splitting field for $S$, then $\fs_r$ consists of the Galois fixed
points of $\fs(\cK_b)_r$.  Note that $\fs_r\ne\fs_{r+}$ implies that
$r\in\bZ\frac{1}{b}$. We remark that the filtration on $\fs$
can be defined in terms of a grading, whose graded pieces we denote by
$\fs(r)$.

\begin{defe} \label{d:fundamental}
\begin{enumerate} 
\item[(i)] A point $x \in \cA$ is called \emph{compatible} (resp. \emph{graded compatible}) with $\fs$ if
$\fs_r=\fg(\cK)_{x,r}\cap \fs$  (resp. $\fs(r)=\fg(\cK)_{x,r}\cap \fs$) for all $r\in\bR$. 
\item[(ii)] A fundamental stratum
  $(x,r,\beta)$ with $x\in\cA$ and $r>0$ is an \emph{$S$-regular stratum} if $x$ is compatible with
  $S$ and $S$ equals the connected centralizer of $\b^\flat$.
  \end{enumerate} 
\end{defe}
In fact, every representative of $\b$ will be regular semisimple with
connected centralizer a conjugate of $S$.

\begin{defe}\label{d:toral}  If $(\cE,\n)$ contains the $S$-regular stratum
  $(x,r,\b)$, we say that $(\cE,\n)$ is \emph{$S$-toral}.
\end{defe}

Recall that the conjugacy classes of maximal tori in $G(\cK)$ are in
one-to-one correspondence with conjugacy classes in the Weyl group $W$
\cite{KL88}.  It turns out that there exists an $S$-toral flat
$G$-bundle of slope $r$ if and only if $S$ corresponds to a regular
conjugacy class of $W$ and $e^{2\pi i r}$ is a regular eigenvalue for
this class~\cite[Corollary 4.10]{BSregular}.  
Equivalently,
$\fs(-r)$ contains a regular semisimple element.

An important feature of $S$-toral flat $G$-bundles is that they can be
``diagonalised'' into $\fs$.  To be more precise, suppose that there
exists an $S$-regular stratum of depth $r$.    Let $\cA(S,r)$ be the open subset of
$\bigoplus_{j\in[-r,0]}\fs(j)$ whose leading component (i.e., the
component in $\fs(-r)$) is regular semisimple.  This is called the set
of \emph{$S$-formal types} of depth $r$.  Let $\Waff_S=N(S)/S_0$ be
the relative affine Weyl group of $S$; it is
the semidirect product of the relative Weyl group $W_S$ and the free
abelian group $S/S_0$.  The group $\Waff_S$ acts on $\cA(S,r)$.  The
action of $W_S$ is the restriction of the obvious linear action while
$S/S_0$ acts by translations on $\fs(0)$.
\begin{thm}{\cite[Corollary 5.14]{BSregular}}\label{formaltypes} If $(\cE,\n)$ is $S$-toral of
  slope $r$, then $\n$ is gauge-equivalent to a connection with matrix
  in $\cA(S,r)\dtt$.  Moreover, the moduli space of $S$-toral flat
  $G$-bundles of slope $r$ is given by $\cA(S,r)/\Waff_S$.
\end{thm}

\section{Formal Coxeter connections} \label{s:Coxeter} 
Recall that $G$ is a simple group over $\bC$, $B$  a Borel subgroup, and  $T$  a maximal torus. Let $\Phi$ denote the set of roots of $G$ and $\Delta\subset \Phi$  the subset  of simple roots. 
Let $\Phi^+$ and $\Phi^-$ denote the set of positive and negative roots, respectively. For each root $\alpha$,  let $\fg_\alpha\subset \fg$ denote the corresponding root subspace. Following \cite[\S 5]{FG}, we choose, once and for all, an ``affine pinning'' for $G$, i.e.,
\begin{enumerate} 
\item  a generator $X_{\alpha}\in \fg_{\alpha}$ for each  negative simple root $\alpha\in -\Delta$; 
\item a generator $E$ for the root space $\fg_\theta$ associated to the highest root $\theta$. 
\end{enumerate} 

 Using this data, we construct  formal connections that turn out to be
 homogeneous toral connections (Definition \ref{d:toral}) associated to
 a certain torus, called the Coxeter torus. For brevity, we refer to
 these connections as formal Coxeter connections. As we shall see, the gauge equivalence class of these connections depends only mildly on the above choices.

\subsection{Recollections on some results of Kostant} The aim of this subsection is to recall some results of \cite{Kostant}, in a format convenient for our purposes. Let $\check{\rho}\in \ft$ denote the sum of fundamental coweights and $x_0:=\check{\rho}/h \in \ft$. Consider the adjoint action of the torsion element $\exp(x_0)=e^{2\pi i x_0}\in T$ on $\fg$.  The resulting eigenspaces define a periodic grading
\begin{equation}\label{eq:grading} 
\fg=\bigoplus_{i\in \bZ/{h\bZ}} \fg_i,
\end{equation} 
where $\fg_0=\ft$, and for $i\in \{1, 2, \cdots, h-1\}$, $\fg_i$ is direct sum of root spaces 
\begin{equation} \label{eq:pieces}
\fg_i=\big(\bigoplus_{ \mathrm{ht}(\alpha)=i-h} \fg_\alpha\big) \oplus \big(\bigoplus_{\mathrm{ht}(\alpha)=i} \fg_\alpha\big).
\end{equation}

Next, let 
\begin{equation}\displaystyle N_1:=\sum_{\alpha\in -\Delta}   X_{\alpha}\in \fg_{-1}
\end{equation}
and $E_1=E$, where we recall that $E$ is a generator of the root space
for the highest root $\theta$.  Then $N_1$ is a principal (=regular)
nilpotent element in $\fg$.  Kostant proved that the element
$N_1+E_1\in \fg$, is regular semisimple, so its centraliser is a
maximal torus, denoted by $S$. Let $\fs:=\on{Lie}(S)$ denote the
corresponding Cartan subalgebra of $\fg$. The grading
\eqref{eq:grading} induces a grading
\begin{equation} 
\fs = \fs_0 \oplus \fs_1\oplus \cdots \oplus \fs_{h-1}.
\end{equation} 
Kostant proved that $\fs_0=0$ and $\fs_i$ contains a regular semisimple element if and only if $\gcd(i,h)=1$, in which case, $\dim(\fs_r)=1$. 

In particular, $N_1+E_1$ is the unique, up to scalar, nonzero element in $\fs_{-1}$. Moreover, if $A_r$ is a (regular semisimple) generator of $\fs_{-r}$,  where $\gcd(r,h)=1$, then there exists generators $X_\alpha$ of root spaces $\fg_\alpha$ comprising $\fg_{-r}$ \eqref{eq:pieces}, and a decomposition $A_r=N_r+E_r$ where 
\begin{equation} \label{eq:NrEr}
N_r=\sum_{\mathrm{ht}(\alpha)=-r} X_\alpha , \quad \quad \quad E_r=\sum_{\mathrm{ht}(\alpha)=h-r} X_\alpha.
\end{equation} 

\begin{exam} \label{e:example} For $\fg=\Sl_n$, let $N_1$ denote the matrix with $1$'s on the subdiagonal and zeros everywhere else, and let $E_1$ be the matrix with a $1$ on the top right hand corner and $0$'s everywhere else. We can take $A_r$ to be $A_1^r$. Then $N_r$ is  the matrix with $1$'s on the $r^\mathrm{th}$ subdiagonal and $0$'s everywhere else, and $E_r$ is the matrix with $1$'s on the $(n-r)^{\mathrm{th}}$ superdiagonal and $0$'s everywhere else. For instance, for $n=5$ we have 
\[
N_2=
 \begin{pmatrix} 
 0 & 0 & 0 & 0 & 0 \\
 0 & 0 & 0 & 0 & 0 \\
1 & 0 & 0 &  0 & 0 \\
0 & 1 & 0 & 0 & 0 \\
0 & 0 & 1 &  0 & 0\\
\end{pmatrix} 
 \qquad \textrm{and} \qquad 
 E_2=
  \begin{pmatrix} 
 0 & 0 & 0 & 1 & 0 \\
 0 & 0 & 0 & 0 & 1 \\
0 & 0 & 0 &  0 & 0 \\
0 & 0 & 0 & 0 & 0 \\
0 & 0 & 0 &  0 & 0\\
\end{pmatrix}. 
 \]
\end{exam} 
 We also remark that given any generators $X_\alpha$ for the root spaces of heights $-r$ and $h-r$, the resulting $N_r+E_r$ can be obtained from an appropriate affine pinning.  Indeed, Kostant shows that a choice of such $X_\alpha$'s is just an affine pinning with respect to a different choice of positive roots.

\subsection{The Coxeter Cartan subalgebra and its graded pieces} We now transport the above results to the setting of loop algebras. Let $x_I$ denote the barycentre  of the fundamental alcove of the standard apartment $\cA$. 
The associated Kac-Moy-Prasad grading is given by
\begin{equation}\label{eq:CoxeterGrading}
\fg(\bC[t,t^{-1}])= \bigoplus_{i\in \bZ} \fg(\cK)_{x_I}(i/h).
\end{equation}
In fact, this is a renormalisation of the \emph{principal grading} of the polynomial loop algebra \cite[\S 14]{Kac}, \cite[Example 3.1]{Chen}, where the degrees have been divided by $h$.

\begin{ntn} By Kostant's theorem, the element  $N_1+t^{-1}E_1\in \fg(\cK)_{x_I}(-1/h)$ is regular semisimple. Following \cite{GKM}, we call its centraliser the \emph{Coxeter torus} and denote it by $\cC$. We call $\fc:=\mathrm{Lie}(\cC)$ the \emph{Coxeter Cartan} subalgebra.  The reason for this terminology is that under the bijection between conjugacy classes of maximal tori in the loop group to conjugacy classes in the Weyl group (cf. \cite{KL88}), the class of $\cC$ maps to the Coxeter class. 
\end{ntn} 

Next, the grading \eqref{eq:CoxeterGrading} induces a grading 
\[
\fc\cap \fg(\bC[t,t^{-1}]) = \bigoplus_{i\in \bZ} \fc(i/h).
\]
This is in fact the canonical Kac-Moy-Prasad grading of $\fc$. In particular, we see that $x_I$ is graded compatible with $\fc$ in the sense of Definition \ref{d:fundamental}. 

\subsubsection{Relationship to the grading \eqref{eq:grading}} As noted in \cite{RY} and \cite{Chen}, there is a  dictionary between Moy-Prasad filtrations and periodic gradings of $\fg$. In the present setting, this means that 
we have a canonical isomorphism between the graded spaces of the principal grading and the grading discussed in the previous subsection, i.e., 
\[
\fg_i\simeq \fg(\cK)_{x_I}({i/h}) \quad \quad i\in \bZ. 
\] 
This observation allows us to transport the results of the previous subsection to the current setting. 
Thus, we find that $\fc({i/h})$ contains a regular semisimple element if and only if $\gcd(i,h)=1$, in which case $\dim(\fc({i/h}))=1$. Moreover,  $N_r+t^{-1}E_r$ is a generator for $\fg(\cK)_{x_I}({-r/h})$ for all $r$ satisfying $1\leq r\leq h$ and $\gcd(r,h)=1$. 

\subsection{Homogeneous $\cC$-toral connections} In the previous subsection, we explained that the canonical Kac-Moy-Prasad grading of the Coxeter Cartan subalgebra $\fc$ is compatible with the point $x_I\in \cA_\bQ$ and that the regular semisimple elements in the graded pieces are exactly of the form $t^{-m}(N_r+t^{-1}E_r)$ with $m, r\in \bZ$, $1\leq r\leq h$ and $\gcd(r,h)=1$. Thus, if we further assume $m\geq 0$, then the strata  
\[
\Big(x_I, \, \, m+r/h, \,\,  t^{-m}(N_r+t^{-1}E_r)\Big)
\]
are, up to rescaling the third entry by an element of $\bC^\times$, the $\cC$-regular strata based at $x_I$ (Definition \ref{d:fundamental}).  
In view of Definition \ref{d:toral}, we find: 

\begin{lem} The homogeneous $\cC$-toral connections based at $x_I$ are the formal connections defined by 
\begin{equation} 
\boxed{\nabla_{r,m}^\lambda :=d + t^{-m}\lambda(N_r+t^{-1}E_r)\frac{dt}{t}}
\end{equation} 
where 
\[
 \lambda\in \bC^\times,\, m\in \bZ_{\geq 0},\, r\in \bZ\cap [1,h), \, \gcd(r,h)=1.  
\]
\end{lem} 

Here, the term homogeneous $\cC$-toral connection means that the $\cC$-formal type (as defined in \S \ref{ss:toral}) is homogeneous. We remark that $N_r$ and $E_r$ are only determined by the affine pinning up to a common nonzero scalar.  Thus, in our notation, we have fixed one such choice. 
(As explained in \S 3.1, an appropriate choice of affine pinning allows one to take $N_r$ and $E_r$ to be any elements of heights $-r$ and $h-r$ with nonzero component in each root space of these heights.) 

\begin{ntn} For brevity, we refer to the $\nabla_{r,m}^\lambda$'s as \emph{formal Coxeter connections} and write $\nabla_{r,m}$ for $\nabla_{r,m}^1$. 
\end{ntn} 

We note that the restriction of the Frenkel-Gross connection $\nabla_{\FG}$ to the formal neighbourhood of the irregular singular point is  $\nabla_{1,0}^{-1}$.  

\begin{prop} \label{p:irregularity} 
The connection $\nabla_{r,m}^\lambda$ is irregular of slope $r/h+m$ with adjoint irregularity 
$(r+mh)\ell$, where $\ell$ is the rank of $G$. 
\end{prop} 
  
  \begin{proof} Since the connection contains a fundamental stratum of depth $m+r/h$, Theorem \ref{t:fundamental} implies that the slope is $m+r/h$. The statement about adjoint irregularity follows from \cite[Lemma 19]{KS}.\footnote{One can also prove these facts using similar considerations to \cite[\S 5]{FG}.}
  
  \end{proof} 
  
  \subsection{Uniqueness} Recall that to define the formal Coxeter connections $\nabla_{r,m}^\lambda$, we fixed an affine pinning of $G$. We will now discuss to what extent the connections depend on this choice.  Let us choose another affine pinning $T'$, $B'$, and $X_\alpha'$, $\alpha\in -\Delta\cup \{\theta\}$. Let $N_1'+E_1'$ denote the  resulting regular semisimple element. By \cite[Theorem 6.2]{Kostant}, there exists $g\in G$ and $\lambda\in \bC^\times$ such that the adjoint action by $g$ takes $T'$ to $T$, $B'$ to $B$, the root spaces with respect to $T'$ to those with respect to $T$, and $X_\alpha'$ to $\lambda X_\alpha$. This implies that 
\begin{equation} 
\mathrm{Ad}_g(N_1'+t^{-1}E_1') = \lambda(N_1+t^{-1}E_1). 
\end{equation} 
Moreover, for each $r$ relatively prime to $h$, we can choose $N_r'$ and $E_r'$ such that the equation
\[
\mathrm{Ad}_g(N_r'+t^{-1}E_r')=\lambda(N_r+t^{-1}E_r)
\]  
holds as well. 

Now, let $\nabla_{r,m}':=d + t^{-m}(N_r'+t^{-1}E_r')\frac{dt}{t}$. The above discussion implies: 

\begin{lem} There exists  $g\in G$ and $\lambda\in \bC^\times$ such that 
\[
\mathrm{Gauge}_g(\nabla_{r,m}') =\nabla_{r,m}^\lambda 
\]
\end{lem} 

It remains to find when two connections $\nabla_{r,m}^\lambda$ and $\nabla_{r,m}^{\lambda'}$ are gauge equivalent. To this end, we determine the  relevant part of the moduli space of $\cC$-toral connections of slope $i/h=r/h+m$ (cf. \cite[Example 18]{KS},  where the case $i=1$ is discussed). 

\subsubsection{Determination of the moduli space} 
By Theorem \ref{formaltypes}, the moduli space of $\cC$-toral connections of slope $i/h$ is 
$\cA(\cC,i/h)/\Waff_\cC$. Here $\cA(\cC,i/h)$ is the open subset of $\fc(-i/h)\oplus\dots\oplus \fc(0)$ for which the component in degree $-i/h$ is regular semisimple.  This subset is nonempty if and only if $i$ is coprime to $h$, in which case its homogeneous degree $-i/h$ part is isomorphic to $\bC^\times$. Recall that $\Waff_\cC$ is the semidirect product of $W_\cC=N_{G(\cK)}(\cC)/\cC$ by $\cC/\cC_0$. Moreover, $\cC/\cC_0$ acts by translation on $\fc(0)$. But the latter space is zero, which implies the action is trivial. Thus, the homogeneous part of  the moduli space is well-defined and isomorphic to $\bC^\times/W_\cC$. 

Finally, $W_\cC$ is
  isomorphic to a subgroup of the centralizer of a Coxeter element in
  $W$~\cite[Proposition 5.9]{BSregular}, so $W_\cC$ is a cyclic
  group of order $h'$ dividing $h$. In fact, in the present case, we have  $h'=h$. To see this, let $\zeta$ be a primitive $h^\mathrm{th}$ root of unity. Kostant \cite{Kostant} proved that there is an element $s\in T$ for which
   \[
  \mathrm{Ad}(s)x=\zeta^k x,\quad \quad \forall \, x\in\frak \fg_\alpha, \quad  |\alpha|=k.
  \]
    In particular, the standard cyclic element $N_1+E_1$ is an eigenvector of $\Ad(s)$ with eigenvalue $\zeta^{-1}$.  It follows that $s\in N(\cC)$ and $\Ad(s)(N_1+t^{-1}E_1)=\zeta^{-1}(N_1+t^{-1}E_1)$, so the image of $s$ in $W_C$ has order $h$. We summarise the above discussion in the following: 
    
\begin{prop} If $r$ is a positive integer coprime to $h$, and $m$ is a nonnegative integer, then  the moduli space of formal Coxeter connections of slope $i/h=m+r/h$ is isomorphic to $\bC^\times/\mu_{h}$.
\end{prop}   

\begin{cor} The formal connections $\nabla_{r,m}^\lambda$ and
  $\nabla_{r,m}^{\mu}$ are gauge equivalent if and only if
  $\lambda/\mu$ is an ${h}^{\mathrm{th}}$ root of unity.
\end{cor} 

\subsection{Local differential Galois group} \label{s:localDifferential}
Let $\mathcal{I}_0$ denote the differential Galois group of $\Spec(\cK)$. Let $\cP_0$ be the subgroup of $\cI_0$ classifying irregular connections. In some references, $\cI_0$ is referred to as the inertia group and $\cP_0$ the wild inertia subgroup.  In view of the Tannakian definition of $\cI_0$, every formal connection $\nabla$ determines a homomorphism $\phi: \cI_0\ra G$. The group $I_0=I_0^\nabla:=\phi(\cI_0)$ is called the differential Galois group of $\nabla$. 

In this subsection, we follow \cite[\S 13]{FG} to determine  $I_0$ for the formal Coxeter connection $\nabla_{r,m}^\lambda$. Let $\phi: \cI_0\ra G$ be the homomorphism associated to $\nabla_{r,m}^\lambda$. Then the image $P_0:=\phi(\cP_0)$  is  the smallest torus $S$ in $G$ whose Lie algebra contains the regular semisimple element $N_r+E_r$. 
The full image $I_0= \phi(\cI_0)$ normalises $S$, and the quotient is generated by the element $n=(2\rho)(e^{\pi i/h})$. As $n$ normalises $S$, it also normalises $T'=\mathrm{C}_{G}(S)$ which is a maximal torus in $G$. The image of $n$ in $N(T')/T'$ is a Coxeter element.  

Next, note that the composition $I_0\rightarrow G\rightarrow \mathrm{Aut}(\fg)$ gives rise to an action of $I_0$ on $\fg$. 
For future reference, we record the following lemma: 
\begin{lem} \label{l:localDifferential}
For formal Coxeter connections $\nabla_{r,m}^\lambda$, we have $\fg^{I_0}=0$.
\end{lem}  
\begin{proof} Indeed, $\fg^{P_0} \subset \fg^{N_r+E_r} = \ft':=\mathrm{Lie}(T')$ implying that $\fg^{I_0} \subset (\ft')^n = 0$. 
\end{proof}

 \section{Rigidity}\label{s:Rigidity}
   Let $X$ be a smooth projective curve over $\bC$. Let $j: U\ra X$ denote the inclusion of a nonempty open subset. 
   \begin{defe} A connection  $\nabla$  on $U$ is said to be (cohomologically) rigid  if 
\[
\mathrm{H}^1 (\mathbb{P}^1, j_{!*}\ad_{\nabla})=0.
\] 
\end{defe} 
As explained in 
\cite[\S 3.2]{YunCDM}, this notion of rigidity is closely related to having no infinitesimal deformations. In more detail, if $\nabla$ satisfies 
\[
\mathrm{H}^i (\mathbb{P}^1, j_{!*}\ad_{\nabla})=0, \qquad i\in \{0,1,2\}
\]
then $\nabla$ has no infinitesimal deformations. In practice,  $\mathrm{H}^0$ and $\mathrm{H}^2$ vanish for many connections of interest, e.g., for irreducible ones. 

\subsection{Main theorem} 
In the previous section, we defined the formal Coxeter connections $\nabla_{r,m}^\lambda$ as connections on the punctured disk.  Now observe that the element of the loop algebra representing this formal connection, i.e. $t^{-m}(N_r+t^{-1}E_r)$, lies in $\fg(\bC[t,t^{-1}])$. Thus, we can think of $\nabla_{r,m}^\lambda$ as a connection on the trivial bundle on $\bGm$. We refer to these global connections simply as \emph{Coxeter connections} and, by an abuse of notation, also denote them by $\nabla_{r,m}^\lambda$. The aim of this section is to determine when $\nabla_{r,m}^\lambda$ is rigid.  
 \begin{thm}\label{t:main} Let $\lambda\in \bC^\times$. A Coxeter connection $\nabla=\nabla_{r,m}^\lambda$ is rigid if and only if we are in one of the following cases: 
 \begin{enumerate} 
  \item[(i)] $m=1$ and $r=1$ in which case $\nabla$ is regular at infinity;  
 \item[(ii)] $m=0$ and $r=1$  in which case $\nabla$ is regular singular at infinity with principal unipotent monodromy\footnote{This is essentially the Frenkel-Gross connection.}; 
 \item[(iii)] $m=0$ and $r$ satisfies the following conditions 
 \[
\begin{tabular}{|c|c|}
\hline
\textrm{Root system} & \textrm{Conditions on $r$} \\
\hline
$A_{n-1}$ & $r|n\pm 1$ \\
\hline
$B_n$ & $r|n+1, \,\, r|2n+1$\\
\hline
$C_n$ & $r|2n\pm 1$\\
\hline
$D_n$ & $r|2n, \, \, r|2n-1$\\
\hline
$E_7$ & $r=7$\\
\hline
\end{tabular}
\]
in which case $\nabla$ is regular singular at infinity with unipotent monodromy $\exp(N_r)$. 
 \end{enumerate} 
 \end{thm} 

For Coxeter connections, it is easy to see that 
\[
\mathrm{H}^0 (\mathbb{P}^1, j_{!*}\ad_{\nabla})=\mathrm{H}^2 (\mathbb{P}^1, j_{!*}\ad_{\nabla})=0.
\]
Thus, being rigid in this case is equivalent to having no infinitesimal deformations. 
 
\subsection{A numerical criterion for rigidity}\label{ss:numerical} Let $\cI=\pi_1^{\mathrm{diff}}(\bGm)$ be the differential Galois group of $\bGm$.  
Then a connection $\nabla$ on $\bGm$ defines a homomorphism $\cI\ra G$ whose image is called the (global) differential Galois group of $\nabla$ and denoted by $I$. It contains the local differential Galois groups $I_0$ and $I_\infty$. 

Now suppose  $\nabla$ is irregular at $0$ and regular singular at $\infty$.  Define 
\begin{equation} \label{eq:numerical} n(\nabla):= \Irr_0(\ad_{\nabla})
  - \dim(\fg^{I_0}) - \dim(\fg^{I_\infty}) + 2\dim(\fg^I);
\end{equation} 
here, $\Irr_0$ denotes the irregularity at $0$.  According to \cite[prop. 11]{FG}, $\dim(\mathrm{H}^1 (\mathbb{P}^1, j_{!*}\ad_{\nabla}))=n(\nabla)$. Thus, 
we find that 
\textrm{$\nabla$ is rigid} if and only if $n(\nabla)=0$.
 We call this the \emph{numerical criterion for rigidity}. 
 
 We now use this criterion to establish when the Coxeter connections $\nabla_{r,m}^\lambda$ are rigid. Since $\lambda$ plays no role in the analysis, we set $\lambda=1$ and omit it from the notation.

\begin{prop} If $\nabla_{r,m}$ is rigid, then $m=0$ or $m=1$. 
\end{prop} 
\begin{proof} 
By Lemma \ref{l:localDifferential}, we have  $\fg^{I_0}=0$. As $I$ contains $I_0$, this implies that $\fg^I$ also vanishes.  Next, by Proposition \ref{p:irregularity},  $ \Irr(\ad_{\nabla_r})= (r+mh)\ell$. Thus, we have 
\[
n(\nabla) = (r+mh)\ell - \dim(\fg^{I_\infty}) \geq (r+mh)\ell - (h+1)\ell= (r+mh-h-1)\ell.
\]
Note that the first inequality is using the fact that $\dim(\fg^{I_\infty})\leq \dim(\fg)=(h+1)\ell$. 
As $m, r\in \bZ_{\geq 0}$ and $h\geq 2$, it is clear that if $n(\nabla)=0$ then either $m=0$ or $m=1$.
\end{proof} 

\begin{prop} The connection $\nabla_{r,1}$ is rigid if and only if $r=1$. 
\end{prop} 
\begin{proof} It is easy to see that the connection is regular at $\infty$, so ${I_\infty}=0$ and $\fg^{I_\infty}=\fg$, and therefore $n(\nabla)=0$. Thus, by the numerical criterion for rigidity, $\nabla$ is rigid. 
\end{proof}

 Next, we treat the case $m=0$ of  the theorem and write $\nabla_r$ for $\nabla_{r,0}$. It is easy to see that the connection $\nabla_{r}$ is regular singular at infinity with monodromy $\exp(N_r)$. Let $\cO_r$ denote the nilpotent orbit containing $N_r$ and $C_\fg(N_r)$ the centraliser of $N_r$ in $\fg$. 
\begin{prop}\label{p:dim}
 The connection $\nabla_{r}$ is rigid if and only if  
\begin{equation}\label{eq:key} \dim(\cO_{N_r})= (h+1-r)\ell;
\end{equation}
or equivalently, 
\begin{equation} \label{eq:key2}
\dim C_\fg(N_r) = r\ell. 
\end{equation} 
\end{prop} 
\begin{proof} We have 
\[
\dim(\fg^{I_\infty})=\dim(C_G(\on{Exp}(N_r)))= \dim(C_\fg(N_r))=\dim(\fg)-\dim(\cO_{N_r}). 
\]
Thus, 
\[
n(\nabla) = r\ell - \dim(\fg^{I_\infty})= (\dim(\fg) - (h+1-r)\ell ) -(\dim(\fg)-\dim(\cO_{N_r}))=\dim(\cO_{N_r}) - (h+1-r)\ell. 
\]
By the numerical criterion, $\nabla$ is rigid if and only if $\dim(\cO_{N_r})=(h+1-r)\ell$. The second equality follows from this and the fact that $\dim(\fg)=(h+1)\ell$. 
\end{proof} 

Note that if $r=1$, then $N_r$ is the principal nilpotent element. Therefore, $\dim(\cO_{N_r})=h\ell$ and \eqref{eq:key} holds. Thus, as proved in \cite{FG}, the connections $\nabla_{1}$ are always rigid. On the other hand, when $r>1$, $\dim(\cO_{N_r})$ depends on the type of the Lie algebra $\fg$; thus, a type by type analysis is required to determine the rigidity of $\nabla_r$. As we have already discussed the case $r=1$, we will not consider it as one of the possibilities.

\subsection{Proof of Theorem \ref{t:main} for classical groups} 
A convenient reference for nilpotent orbits in semisimple Lie algebras is \cite{CM}. In particular, in \S 5 of \emph{op. cit}, it is shown that Jordan canonical form leads to a bijection between nilpotent orbits in classical groups and certain partitions. This in turn leads to formulae for the dimension of nilpotent orbits in terms of the dual partition (cf. appendix of \cite{MY}). Below we will use these formulae to determine rigidity of $\nabla_r$ for classical groups. .

\subsubsection{Type $A_{n-1}$} Let us write $n=kr+n'$ where $0\leq n'<r$.  Since $N_1$ is principal nilpotent and one can take $N_r=N_1^r$, it is easy to see that the Jordan form of $N_r$ has $n'$ parts of size $k+1$ and $r-n'$ parts of size $k$. Thus, the dual partition has $k$ parts of size $r$ and $1$ part of size $n'$ and 
\[
\dim(\cO_{N_r}) = n^2 - kr^2 - (n')^2. 
\]
The equality \eqref{eq:key2} therefore takes the form 
\[
kr^2 + (n')^2 -1= r(n-1).
\]
Using the equality $n=kr+n'$, we find that the above equality holds if and only if
\[
(n'-1)(r-(n'+1))=0.
\]
Thus, either $n'=1$ or $n'=r-1$. This is in turn equivalent to $r| (n-1)$ or $r|(n+1)$.

\subsubsection{Type $B_n$} Nilpotent orbits in  $\mathfrak{so}_{2n+1}$ are in bijection with those partitions of $2n+1$ where even parts appear with even multiplicity.  
Let us write $2n+1=kr+n'$ where $0\leq n'<r$. In this case, $N_1$ is also principal nilpotent viewed as an element of $\Sl_{2n+1}$, and $N_r$ can again be taken to be $N_1^r$. Accordingly, the partition associated to the nilpotent element $N_r$ has $n'$ parts of size $k+1$ and $r-n'$ parts of size $k$. To see that this partition satisfies the required constraint, note that $r$ is odd, since it is coprime to the Coxeter number $2n$. Now if $n'$ is even, then we must have that $kr$ is odd, thus $k$ is odd, which means that $k+1$ is even. On the other hand, if $n'$ is odd, then $r-n'$ and $k$ are even.

The dual partition has $k$ parts of size $r$ and $1$ part of size $n'$. Therefore, we have 
\[
\dim C_\fg({N_r})=\frac{1}{2}\left( kr^2 + (n')^2 -
\begin{cases} 
n' & \textrm{if $k$ is even}\\
r-n' & \textrm{if $k$ is odd}
\end{cases} 
\right)
\]
Equating this with $r(2n)=\frac{1}{2} r(kr+n'-1)$, equality \eqref{eq:key2} takes the form 
\[
(n')^2 -
\begin{cases} 
n' & \textrm{if $k$ is even}\\
r-n' & \textrm{if $k$ is odd}
\end{cases} 
=r(n'-1). 
\]

If $k$ is even, then we obtain $n'=1$ or $r=n'$. Both cases are
impossible. Indeed, $n'=1$ implies that $r|2n$ and that contradicts
$\gcd(r,2n)=1$. The case  $r=n'$ is not allowed because $r$ is assumed to be less than $n'$. On the other hand, if $k$ is odd, we obtain that either $n'=0$ or $r=n'+1$. In the first case, we obtain $r|(2n+1)$  and in the second case $r|(n+1)$. 

Finally we check that if $r|(2n+1)$ or $r|(n+1)$ then $k$ cannot be even. If $r|2n+1$, then $n'=0$, so $2n+1=kr$, thus $k$ is odd. If $rs=n+1$, and $k=2t$, then $2n+1=2rs-1=2tr+n'$ implying that $n'=2r(s-t)-1$.  But now $n'\ge 0$ implies $s>t$.  Then $n'\ge 2r-1$ which is impossible since $n'<r$.

\subsubsection{Type $C_n$} Nilpotent orbits in $\mathfrak{sp}_{2n}$ are in bijection with partitions of $2n$ in which odd parts appear with even multiplicity. Let us write $2n=kr+n'$ with $0< n'<r$. Recall that $\gcd(r,2n)=1$; in particular, $r$ is odd.  As in types $A$ and $B$, the $N_r$'s are powers of a single Jordan block.  Thus, the nilpotent element  $N_r$ corresponds to the partition of $2n$ consisting of $n'$ parts of size $k+1$ together with $r-n'$ parts of size $k$. The fact that this partition satisfies the desired constraint follows by similar considerations to those in type $B_n$. 

The dual partition has $k$ parts of size $r$ and $1$ part of size $n'$. 
Now, we have 
\[
\dim C_\fg({N_r})=
\frac{1}{2}\left(kr^2 + n'^2 + 
\begin{cases} 
n' & \textrm{if $k$ is even}\\
r-n' & \textrm{if $k$ is odd}
\end{cases} 
\right).
\]
Equating this with $rn=r(kr+n')/2$, we find that the equality \eqref{eq:key2} takes the form 
\[
(n')^2 + 
\begin{cases} 
n' & \textrm{if $k$ is even}\\
r-n' & \textrm{if $k$ is odd}
\end{cases} 
=rn'. 
\]
If $k$ is even, we have 
\[
(n')^2+n' = rn' \implies r=n'+1 \implies 2n=kr+n'=kr+r-1 \implies 2n+1=r(k+1) \implies r|(2n+1).
\]
This is indeed consistent with our assumption because $2n+1=r(k+1)$ implies $k$ is even. 

On the other hand, if $k$ is odd, we obtain 
\[
(n')^2+r-n'=rn' \implies (n')^2-n'=(n'-1)r \implies \textrm{$r=n'$ or $n'=1$}.
\]
Note, however, that the case $r=n'$ is impossible. Thus, we obtain, 
\[
2n=kr+n'=kr+1 \implies 2n-1=kr \implies r|(2n-1).
\]
This is also consistent with our assumption because $2n-1=kr$ implies that $k$ has to be odd. 

Finally, note that 
if $r|2n-1$, then $2n=rs+1$, so $s=k$.  Now $r$ is odd and therefore so is  $k$. On the other hand, 
if $r|2n+1$, then $2n=rs-1$, so $k=s-1$, i.e. $2n=r(s-1)+(r-1)$.
Since $r-1$ is even and $r$ is odd, $s-1$ must be even, i.e. $k$ is even.

\subsubsection{Type $D_n$} Nilpotent orbits in  $\mathfrak{so}_{2n}$ are in bijection with partitions of $2n$ in which even parts appear with even multiplicity, except that very even partitions, i.e., those with only even parts each having even multiplicity, correspond to two nilpotent orbits.\footnote{These two orbits are identified if one uses $\Or_{2n}$ instead of $\SO_{2n}$.}  The Coxeter number is $2n-2$ and by assumption $\gcd(r,2n-2)=1$. Let us write $2n=kr+n'+1$ with $0\leq n'<r$. Here, $N_1$ is no longer a single Jordan block.  Instead, it is subregular as an element of $\Sl_{2n}$, i.e., it corresponds to the partition $(2n-1,1)$.  It is easy to check that one can take $N_r=N_1^r$.  It follows that the partition associated to $N_r$ has $n'$ parts of size $k+1$, $r-n'$ parts of size $k$, and $1$ part of size $1$. The dual 
partition has $1$ part of size $r+1$, $1$ part of size $n'$, and $k-1$ parts of size $r$. 
Therefore, we have
\[
\dim C_\fg({N_r})=\frac{1}{2}\left((r+1)^2+(k-1)r^2 +(n')^2 - 
\begin{cases} 
n'+1 & \textrm{if $k$ is even}\\
r-n'+1 & \textrm{if $k$ is odd}
\end{cases} 
\right).
\]
Equating this with $rn=\frac{1}{2}r(n'+kr+1)$, we find that the equation \eqref{eq:key2} takes the form 
\[
(n')^2 - 
\begin{cases} 
n' & \textrm{if $k$ is even}\\
r-n' & \textrm{if $k$ is odd}
\end{cases} 
=r(n'-1). 
\]
If $k$ is even, we obtain that either $n'=1$ or $n'=r$. Neither of these is a possibility because the first case contradicts $\gcd(2n-2,r)=1$ and the second contradicts $r<n'$. On other hand, if $k$ is odd, we obtain that $n'=0$ or $r=n'+1$, which implies $r|(2n-1)$ or $r|(2n)$. 

Finally note that in the first case, $n'=0$ and $k|(2n-1)$,  so $k$ is odd. In the second case, $n'=r-1$, thus $2n=r(k+1)$, implying that $k$ is odd.

\subsection{Proof of Theorem \ref{t:main} for exceptional groups} 
We now investigate the rigidity of $\nabla_r$ for exceptional groups. 
\subsubsection{Type $G_2$}  
We have  $h=6$; thus, the conditions $\gcd(r,h)=1$ and $1<r<6$ imply that $r=5$ is the only case to consider. By Proposition \ref{p:dim}, for $\nabla_r$ to be rigid, we need $\dim(\cO_{N_r})=4$. However, by inspecting the table  for  nilpotent orbits in type $G_2$  \cite[\S 8]{CM}, one finds that $\fg$ has no nilpotent orbit of dimension $4$. Thus,  $\nabla_5$ is not rigid. 

\subsubsection{Type $F_4$} In this case, $h=12$. As $\gcd(r,12)=1$, we see that possibilities for $r$ are $5$, $7$, and $11$. For $\nabla_r$ to be rigid, the dimension of the orbit containing $N_r$ must be $32$, $24$, and $8$, respectively. But $\fg$ has no nilpotent orbits of these sizes \cite{CM}, so the $\nabla_r$'s are not rigid. 

\subsubsection{Type $E_6$} In this case, $h=12$; thus $r\in \{5, 7, 11\}$. For $\nabla_r$ to be rigid, the dimension of the orbit containing $N_r$ must be $48$, $36$, and $12$ respectively. There are no nilpotent orbits of size $12$ or $36$. There is, however, a nilpotent orbit of size $48$. Table 4 in \cite{deGraaf} (last line of page 8), however, shows that the nilpotent orbit containing $N_5$ has label $2A_1+A_2$. According to 8.4 of \cite{CM}, this nilpotent orbit has size $50$. Thus, $\nabla_5$ is not rigid.

\subsubsection{Type $E_7$}
In this case $h=18$; thus, $r\in \{5, 7, 11, 13, 17\}$. For $\nabla_r$ to be rigid, the dimension of the orbit containing $N_r$ must be $98, 84, 56, 42$, and $14$ respectively. Note that in this case, we do not have a nilpotent orbit of size $14$, $42$, or $56$. Moreover, as W. de Graaf explained to us, using a GAP computation, one can show that $\dim(\cO_{N_5})=100$. In fact, this nilpotent orbit has Dynkin representative 
\[
[0,0,0,0,2,0,0]
\]
and is denoted by $A_3+A_2+A_1$ (see page 10 of \cite{deGraaf}). Thus $\nabla_5$ is not rigid. On the other hand, the table in \cite{deGraaf} shows that in fact $\dim(\cO_{N_7})=84$, as required. This is the orbit $A_2+3A_1$ with Dynkin representative 
\[
[2, 0,0,0,0,0,0].
\]
  Thus, $\nabla_7$ is rigid.

\subsubsection{Type $E_8$} 
In this case, $h=30$; thus, $r\in \{7, 11, 13, 17, 19, 23, 29\}$. For $\nabla_r$ to be rigid, the dimension of the orbit containing $N_r$ must be $192$, $160$, $144$, $112$, $96$, $64$, $16$,  respectively. Amongst these, only $192$ and $112$ are dimensions of some nilpotent orbits. However, one can check \cite{deGraaf} that the dimension of $\dim(\cO_{N_7})= 196$ and $\dim(\cO_{N_{17}})= 128$. Thus, the $\nabla_r$'s are not rigid.

\subsection{Global differential Galois group} In \S \ref{s:localDifferential}, we gave a description of the local differential Galois group $I_0$ of $\nabla_{r,m}$. In this subsection, we discuss the global differential Galois group $I$.
It is easy to see that the proof of \cite[Proposition 8]{FG} generalises to our setting to show that $I$ is reductive. 
Moreover,  $I$ contains $I_0$ and $I_\infty$ as subgroups.  Recall
that $I_\infty$ is generated by the monodromy $\exp(N_r)$ if $m=0$ and
is trivial otherwise.  In the case of $\nabla_1$, the monodromy is the principal unipotent class. This puts
a severe restriction on the reductive group $I$ and leads to the
classification of $I$ in each type \cite[\S 13]{FG}. However, when
$r>1$, the unipotent monodromy $\exp(N_r)$ of $\nabla_r$ is not the principal class;
thus, we do not have such a strong restriction on the type of the
reductive group $I$. For instance, as W. de Graaf explained to us,
there are are at least 46 semisimple subalgebras of $E_7$ of different
types containing the nilpotent element $N_7$.

Nevertheless, we expect that the global differential Galois groups of the
$\nabla_{r,m}$'s are ``large'', i.e., almost equal to $G$. Our
expectation is motivated by the following computation of the
differential Galois groups of the $\GL_n$-analogue of these
connections. In type $A_{n-1}$, consider the connection
$\mathrm{std}_{\nabla_{r,m}}$ associated to the standard representation.
Note that this is just the trivial vector bundle of rank $n$ on $\bGm$
equipped with the connection $\nabla_{r,m}$.

\begin{prop} The differential Galois group of
  $\mathrm{std}_{\nabla_{r,m}}$ is equal to $\SL_n(\bC)$ if $n$ is odd and
  to $\SP_{n}(\bC)$ if $n$ is even. 
\end{prop}

\begin{proof}  
  Let $I$ denote the differential Galois group of
  $\mathrm{std}_{\nabla_{r,m}}$.  Let $I^\circ$ denote the identity
  component of $I$ and ${I^\circ}'$ the derived subgroup of $I^\circ$.
  The connection $\nabla_{r,m}$ is irreducible as its restriction to $0$
  is irreducible. Moreover, as $\gcd(r,n)=1$ and the irregularity is
  an integer, the slope $m+\frac{r}{n}$ in the slope decomposition at
  the irregular point must appear with multiplicity $n$. Therefore, we
  are in a position to apply Theorem 2.8.1 of
  \cite{Katzbook}\footnote{This result depends on Gabber's ``torus
    trick'' \cite[\S 1]{Katzbook}, which appears to be a type $A$
    phenomenon.} to conclude that ${I^\circ}'$ equals $\SL_n$ if $n$
  is odd, and it equals one of $\SL_n$, $\SP_n$ or $\SO_n$ if $n$ is
  even.

Next, observe that the defining matrix $t^{-m}(N_r+t^{-1}E_r)$ is in $\Sl_n(\C[t,t^{-1}])$. Thus,
by \cite[Proposition 1.31]{SV}, we have $I\subseteq \SL_n$. Now for
odd $n$, the previous paragraph implies that $I\supseteq \SL_n$ and so
$I=\SL_n$. For even $n$, the nontrivial pinned
automorphism $\sigma$ of $\mathfrak{sl}_n$ fixes $N_r$ and
$E_r$. Thus, $I$ is a subgroup of $\SL_n^\sigma=\SP_n$. In view of the
previous paragraph, we have $I=\SP_n$.
\end{proof}

\begin{rem} The Coxeter connection $\nabla_{1,1}$ associated to $G=G_2$ appears in  \cite[\S 2.10.6]{Katzbook}.\footnote{We thank the referee for bringing this to our attention.}  Indeed, a computation shows that the realisation of $\nabla_{1,1}$ in the seven-dimensional representation of $G_2$ is given by the differential equation 
\begin{equation}
 \partial^7 - 2t \partial -t,
\end{equation} 
where $\partial=t\frac{d}{dt}$. Moreover, Katz proved that the global differential Galois of the above differential equation is $G_2$. 
\end{rem}

  \end{document}